\documentclass[a4paper,reqno]{amsart}

\usepackage{amssymb}
\usepackage{latexsym}
\usepackage{amsmath}
\usepackage{amsthm}
\usepackage{euscript}
\usepackage{bbm}

\def\cal H{{\mathcal H}}

\def\C{\mathbb{C}}

\def\N{\mathbb{N}}

\def\dom{{\text{\rm dom\,}}}

\def\dist{{\text{\rm dist}}}

\def\phi{\varphi}

\DeclareMathOperator{\Real}{Re}

\DeclareMathOperator{\spann}{span}

\DeclareMathOperator{\diam}{diam}

\renewcommand{\theta}{\vartheta}

\newtheorem{theorem}{Theorem}[section]
\newtheorem*{thm*}{Theorem}
\newtheorem{proposition}[theorem]{Proposition}
\newtheorem{corollary}[theorem]{Corollary}

\theoremstyle{definition}

\newtheorem{example}[theorem]{Example}

\newtheorem*{ack}{Acknowledgement}

\numberwithin{equation}{section}

\title[Eigenvalue estimates for the Laplacian on a metric tree]{Eigenvalue estimates for the Laplacian on a metric tree}

\author[J.~Rohleder]{Jonathan Rohleder}
\address{TU Hamburg \\ Institut f\"ur Mathematik \\
Am Schwarzenberg-Campus~3 \\
Geb\"aude E \\
21073 Hamburg \\
Germany}
\email{jonathan.rohleder@tuhh.de}

\begin{document}

\begin{abstract}
We provide explicit upper bounds for the eigenvalues of the Laplacian on a finite metric tree subject to standard vertex conditions. The results include estimates depending on the average length of the edges or the diameter. In particular, we establish a sharp upper bound for the spectral gap, i.e.\ the smallest positive eigenvalue, and show that equilateral star graphs are the unique maximizers of the spectral gap among all trees of a given average length.
\end{abstract}

\maketitle

\section{Introduction}

Spectral theory of differential operators on metric graphs, so-called quantum graphs, has been a very active field of research in recent years, see, e.g., the monograph~\cite{BK13}. Amongst other topics, particular attention was paid to eigenvalue inequalities for Laplace and Schr\"odinger operators on metric graphs. For a selection of recent contributions in this field we refer the reader to~\cite{BK12,DR16,D15,DH10,EFK11,EJ12,KKT16,LP08,NS00} and the references therein. An often-considered model operator is the standard (sometimes also called free or Kirchhoff) Laplacian $- \Delta_G$ on a finite metric graph~$G$, i.e., the (negative) second derivative operator in $L^2 (G)$ subject to standard matching conditions at all vertices; cf.\ Section~\ref{sec:preliminaries} for the details. The selfadjoint operator $- \Delta_G$ and the properties of its purely discrete spectrum have been studied intensively and there has been considerable interest in explicit bounds for the eigenvalues 
\begin{align*}
 0 = \lambda_1 (G) < \lambda_2 (G) \leq \lambda_3 (G) \leq \dots
\end{align*}
of $- \Delta_G$ and, particularly, for the {\em spectral gap}, i.e., the first positive eigenvalue~$\lambda_2 (G)$, see~\cite{F05,K15,KMN13,KN14,N87}. In the recent paper~\cite{KKMM16} the estimate 
\begin{align}\label{eq:KKMMboundIntro}
 \lambda_2 (G) \leq \frac{|E|^2 \pi^2}{L (G)^2}
\end{align}
was established, where $|E|$ denotes the number of edges and $L (G)$ is the total length of $G$, and the class of maximizers of $\lambda_2 (G)$ among all graphs of a given average length $L (G) / |E|$ is specified. Moreover, it was shown in~\cite{KKMM16} that an upper estimate for the eigenvalues of $- \Delta_G$ in terms of the diameter of $G$ only cannot hold. 

In the present note we consider an important subclass of graphs, the class of trees, i.e., graphs without cycles. Our main aim is to expose two peculiarities of this class of graphs with respect to which they differ essentially from the general case. First, we prove that within the class of trees the estimate~\eqref{eq:KKMMboundIntro} can be improved and that, in fact,
\begin{align}\label{eq:averageIntro}
 \lambda_{k + 1} (G) \leq \frac{k^2 |E|^2 \pi^2}{4 L (G)^2}
\end{align}
holds for all $k \in \N$ if $G$ is any finite tree with $|E| \geq 2$. For the spectral gap of a tree this implies
\begin{align}\label{eq:gapToll}
 \lambda_2 (G) \leq \frac{|E|^2 \pi^2}{4 L (G)^2}.
\end{align}
Moreover, we show that equality in~\eqref{eq:gapToll} holds if and only if $G$ is an equilateral star graph. In other words, the equilateral star graphs are the unique maximizers of the spectral gap among all trees of a given average length.

Second, we observe that trees---in contrast to general graphs---admit an upper bound for the eigenvalues of $- \Delta_G$ in terms of the diameter only, namely,
\begin{align}\label{eq:diamIntro}
 \lambda_{k + 1} (G) \leq \frac{k^2 \pi^2}{\diam (G)^2}
\end{align}
for all $k \in \N$. This is related to the fact that the diameter of a tree is given by the length of a path which connects two boundary vertices. The proofs of both estimates~\eqref{eq:averageIntro} and~\eqref{eq:diamIntro} rely on a domain monotonicity principle for the eigenvalues of the free Laplacian on a graph.

In addition to these results we provide an eigenvalue bound of slightly different nature: We compare the eigenvalues of the standard Laplacian on a tree $G$ with those of the Dirichlet Laplacian on $G$, i.e.\ the Laplacian subject to a Dirichlet condition at each vertex. This leads to the estimate
\begin{align}\label{eq:soso}
 \lambda_{k + 1} (G) \leq \lambda_k^{\rm D}, \quad k \in \N,
\end{align}
for any finite tree $G$, where $\lambda_1^{\rm D} \leq \lambda_2^{\rm D} \leq \dots$ are the eigenvalues of the Dirichlet Laplacian. This can be viewed as a graph counterpart of the inequality between Dirichlet and Neumann eigenvalues of the Laplacian on a Euclidean domain found in~\cite{F91}, see also~\cite{F04}. We remark that the estimate~\eqref{eq:soso} can also be derived from Lemma~4.5 in~\cite{BBW15} with a totally different proof. However, the proof provided in the present paper is a little more direct and allows, in addition, to characterize the class of trees for which the estimate~\eqref{eq:soso} is strict for all $k$. Note that the eigenvalues of the Dirichlet Laplacian depend only on the edge lengths but are independent of the geometry of $G$ and can be computed explicitly. Therefore the estimate~\eqref{eq:soso} leads to explicit estimates for the eigenvalues of $- \Delta_G$. As an immediate consequence we get
\begin{align}\label{eq:juche}
 \lambda_{k + 1} (G) \leq \frac{k^2 \pi^2}{L_{\max}^2}, \quad k \in \N,
\end{align}
for any finite tree $G$, where $L_{\max}$ is the length of the longest edge in $G$. We provide an example which shows that the bound~\eqref{eq:juche} is sharp for the spectral gap, i.e.\ for $k = 1$. Our proof of~\eqref{eq:soso} is based on the construction of appropriate test functions similar to those used in the PDE case in~\cite{F04}.

\begin{ack}
The author wishes to thank Gregory Berkolaiko for drawing his attention to the eigenvalue inequalities contained in~\cite{BBW15}. Moreover, the author gratefully acknowledges financial support by the Austrian Science Fund (FWF), project P~25162-N26.
\end{ack}

\section{Preliminaries}\label{sec:preliminaries}

In this section we provide preliminaries on Laplacians on metric graphs; for more details we refer the reader to the monograph~\cite{BK13}. A finite metric graph is a discrete graph $G$ consisting of a finite set of vertices $V$ and a finite set of edges $E$ which is equipped, in addition, with a length function $L : E \to (0, \infty)$, which assigns a length to each edge. Any edge $e \in E$ is then identified with the interval $[0, L (e)]$, and this way of choosing coordinates induces a metric on $G$. In the following we simply write that $G$ is a graph and mean a metric graph. We denote the total length of a graph by $L (G) = \sum_{e \in E} L (e)$ and say that a graph $G$ is equilateral if the length function $L$ is constant on $E$. For any $e \in E$ we write $o (e)$ for the vertex from which the edge $e$ originates and $t (e)$ for the vertex at which $e$ terminates. Throughout this paper we will assume for simplicity that $G$ is connected, that is, each two vertices in $G$ are connected by a path; the obvious analogs of the results for disconnected graphs hold as well. Moreover, in most of what follows we assume that $G$ is a tree, i.e., $G$ does not contain cycles or, equivalently, $|V| = |E| + 1$. By a boundary vertex we mean a vertex of degree one; accordingly, a boundary edge is an edge which is incident to a boundary vertex. Vertices of degree two or larger are called interior vertices. 

On a finite, connected graph $G$ we denote by $L^2 (G)$ the space of square-integrable functions $f : G \to \C$, equipped with the standard norm and inner product. For $f \in L^2 (G)$ we denote by $f_e$ the restriction of $f$ to an edge $e \in E$, identified with a function on $[0, L (e)]$. Moreover, we consider the Sobolev spaces
\begin{align*}
 \widetilde H^k (G) = \bigoplus_{e \in E} H^k (e), \quad k = 1, 2, \dots,
\end{align*}
and say that a function $f \in \widetilde H^1 (G)$ is continuous on $G$ if and only if the limits of $f_e$ and $f_{\hat e}$ towards a vertex $v$ coincide whenever $e$ and $\hat e$ are edges incident to $v$. Accordingly we define
\begin{align*}
 H^1 (G) = \big\{ f \in \widetilde H^1 (G) : f~\text{is~continuous~on}~G \big\}.
\end{align*}
For $f \in H^1 (G) \cap \widetilde H^2 (G)$ we use the abbreviation
\begin{align*}
 \partial_\nu f (v) = \sum_{t (e) = v} f_{e}' (L (e)) - \sum_{o (e) = v} f_e' (0)
\end{align*}
for the sum of the normal derivatives of the edges incident to a vertex $v \in V$.

The object of our interest is the {\em standard Laplacian} $- \Delta_G$ in $L^2 (G)$ defined by
\begin{align*}
 (- \Delta_G f)_e & = - f_e'', \quad e \in E, \\
 \dom (- \Delta_G) & = \big\{ f \in H^1 (G) \cap \widetilde H^2 (G) : \partial_\nu f (v) = 0~\text{for~all}~v \in V \big\}.
\end{align*}
The operator $- \Delta_G$ is selfadjoint in $L^2 (G)$ and its spectrum consists of isolated, nonnegative eigenvalues with finite multiplicities. A nondecreasing enumeration of all eigenvalues, counted with multiplicities, is given by the min-max principle
\begin{align}\label{eq:minmax}
 \lambda_k (G) = \min_{\substack{F \subset H^1 (G) \\ \dim F = k}} \max_{\substack{f \in F \\ f \neq 0}} \frac{\int_G |f'|^2 d x}{\int_G |f|^2 d x}, \quad k = 1, 2, \dots;
\end{align}
cf.~\cite[Chapter~3]{BK13}. Note that in the formula~\eqref{eq:minmax} we have $\lambda_1 (G) = 0$, which is obtained by plugging in any constant function $f$ on~$G$. For later use we also point out the following consequence of~\eqref{eq:minmax}: Denote by $N_G (\iota)$ the number of eigenvalues of $- \Delta_G$, counted with multiplicities, within the real interval $\iota$, that is,
\begin{align*}
 N_G (\iota) = \# \left\{k \in \N : \lambda_k (G) \in \iota \right\}.
\end{align*}
Then
\begin{align}\label{eq:numberEV}
 N_G ([0, \mu]) = \max \Big\{\!\dim F : F \subset H^1 (G), \int_G |f'|^2 d x \leq \mu \int_G |f|^2 d x, f \in F \Big\}
\end{align}
holds for any $\mu \geq 0$.

In the following we shortly visit two basic classes of trees and the spectra of the corresponding standard Laplacians; they will play a role later.

\begin{example}\label{ex:path}
A {\em path graph} is a connected tree $G$ which contains only vertices of degrees one and two. Due to the standard vertex conditions of each $f \in \dom (- \Delta_G)$ the Laplacian $- \Delta_G$ on a path graph $G$ can be identified with the operator $- \frac{d^2}{d x^2}$ on an interval of length $L (G)$ subject to Neumann boundary conditions at both endpoints. Thus the nonzero eigenvalues of $- \Delta_G$ are given by
\begin{align*}
 \lambda_{k + 1} (G) = \frac{k^2 \pi^2}{L (G)^2}, \quad k \in \N,
\end{align*}
for any path graph $G$.
\end{example}

\begin{example}\label{ex:star}
A {\em star graph} is a tree $G$ with a vertex $v_0$ such that each edge $e \in E$ is incident to $v_0$. We assume $|E| \geq 2$ for any star graph. For an equilateral star graph with constant edge length $L = L (e)$ for each $e \in E$, a simple calculation shows that the spectrum of $- \Delta_G$ consists of the numbers
\begin{align*}
 \frac{m^2 \pi^2}{L^2} \quad \text{and} \quad \frac{(m + \frac{1}{2})^2 \pi^2}{L^2}, \quad m = 0, 1, 2, \dots,
\end{align*}
where each eigenvalue of the latter form has the multiplicity $|E| - 1$ and each eigenvalue of the form $\frac{m^2 \pi^2}{L^2}$ has multiplicity one. For the eigenvalues counted with multiplicities this means
\begin{align*}
 \lambda_{|E| j + 1} (G) = \frac{j^2 \pi^2}{L^2} \quad \text{and} \quad \lambda_{|E| j + 2} (G) = \dots = \lambda_{|E| j + |E|} (G) = \frac{(j + \frac{1}{2})^2 \pi^2}{L^2},
\end{align*}
$j = 0, 1, 2, \dots$ In particular, the spectral gap is given by $\lambda_2 (G) = \frac{\pi^2}{4 L^2} = \frac{|E|^2 \pi^2}{4 L (G)^2}$.
\end{example}

\section{Estimates involving the average length and the diameter}

In this section we provide estimates for the eigenvalues of the standard Laplacian on a tree in terms of the arithmetic mean of the edge lengths and in terms of the diameter. As a special case it turns out that among all trees of a given average length the equilateral star graphs are the unique maximizers of the spectral gap.

The proofs of the main results of this section are based on the following domain monotonicity principle. Following the terminology used in~\cite{KKMM16} we say that a {\em pendant graph} is attached to a given graph $H$ if an additional finite, connected graph is attached to precisely one vertex of $H$. We remark that the special case $\lambda_2 (G) \leq \lambda_2 (H)$ of the following proposition can also be derived from~\cite[Theorem~2]{KMN13} or~\cite[Lemma~2.3]{KKMM16}. 

\begin{proposition}\label{prop:monotonicity}
Let $H$ be a finite, connected graph and let $G$ be the graph which arises from attaching pendant graphs to vertices of $H$. Then
\begin{align}\label{eq:monotone}
 \lambda_k (G) \leq \lambda_k (H), \quad k \in \N.
\end{align}
If $H$ is a path graph and one of the boundary vertices of $H$ appears as an interior vertex in $G$ then the inequality is strict for all $k \geq 2$. 
\end{proposition}

\begin{proof}
For $k = 1$ the inequality~\eqref{eq:monotone} is trivially satisfied. Let $k \geq 2$ and $\mu = \lambda_k (H) > 0$. Moreover, define $F = \spann \bigcup_{\lambda \leq \mu} \ker (- \Delta_H - \lambda)$. Then
\begin{align}\label{eq:usualEst}
 \int_H |f'|^2 d x \leq \mu \int_H |f|^2 d x, \quad f \in F.
\end{align}
For each $f \in F$ let $\widetilde f : G \to \C$ be such that $\widetilde f |_H = f$ and $\widetilde f = f (v)$ identically on the pendant graph attached to a vertex $v$ of $H$, if any. Then $\widetilde f \in H^1 (G)$, the derivative $\widetilde f'$ vanishes on $G \setminus H$ and, hence,
\begin{align}\label{eq:fastFertig}
 \int_G |\widetilde f'|^2 d x = \int_H |f'|^2 d x \leq \mu \int_H |f|^2 d x \leq \mu \int_G |\widetilde f|^2 d x
\end{align}
by~\eqref{eq:usualEst}. As $\dim F \geq k$,~\eqref{eq:monotone} follows with the help of the min-max principle~\eqref{eq:minmax}.

Let us now assume that $H$ is a path graph and that there exists a boundary vertex $v$ of $H$ which is an interior vertex of $G$.  We claim that the inequality~\eqref{eq:fastFertig} is strict for each nontrivial $f \in F$. In fact, if we assume that for some nontrivial $f \in F$ we have $\int_G |\widetilde f'|^2 d x = \mu \int_G |\widetilde f|^2 d x$ then in particular the last inequality in~\eqref{eq:fastFertig} must be an equality, which implies $\widetilde f |_{G\setminus H} = 0$ and, in particular, $f (v) = 0$. Moreover, $f'$ vanishes on $v$ as $f \in \dom (- \Delta_H)$ and $v$ is a boundary vertex. As $- \widetilde f'' = \mu \widetilde f$ on the path graph $H$, which can be identified with a single interval, it follows $f = 0$ identically on $H$, a contradiction. Thus $\int_G |\widetilde f'|^2 d x < \mu \int_G |\widetilde f|^2 d x$ for all nontrivial $f \in F$ and the min-max principle implies the second assertion. 
\end{proof}

With the help of Proposition~\ref{prop:monotonicity} we can prove the following theorem. Its proof extends an idea in the proof of~\cite[Theorem~4.2]{KKMM16}.

\begin{theorem}\label{thm:average}
Let $G$ be a finite, connected tree with $|E| \geq 2$. Then
\begin{align}\label{eq:yippieh}
 \lambda_{k + 1} (G) \leq \frac{k^2 |E|^2 \pi^2}{4 L (G)^2}, \quad k \in \N.
\end{align}
Moreover, for $k = 1$ equality holds in~\eqref{eq:yippieh} if and only if $G$ is any equilateral star; for $k \geq 2$ equality holds in~\eqref{eq:yippieh} if and only if $G$ is an equilateral star with $|E| = 2$.
\end{theorem}

\begin{proof}
Let $e_1, e_2$ be two distinct edges in $G$ such that $e_1 \geq e_2 \geq e$ holds for each $e \in E \setminus \{e_1\}$ and let $H$ be the unique path through $G$ which starts with $e_1$ and terminates with $e_2$. Note that the arithmetic mean of $L (e_1)$ and $L (e_2)$ is greater or equal to $L (G) / |E|$ and thus
\begin{align}\label{eq:einmal}
 \lambda_{k + 1} (H) = \frac{k^2 \pi^2}{L (H)^2} \leq \frac{k^2 \pi^2}{( L(e_1) + L (e_2) )^2} \leq \frac{k^2 |E|^2 \pi^2}{4 L (G)^2},
\end{align}
see Example~\ref{ex:path}. Observe that $G$ can be obtained by attaching pendant graphs to the vertices of $H$ and thus Proposition~\ref{prop:monotonicity} together with~\eqref{eq:einmal} yields
\begin{align}\label{eq:undNochmal}
 \lambda_{k + 1} (G) \leq \lambda_{k + 1} (H) \leq \frac{k^2 |E|^2 \pi^2}{4 L (G)^2},
\end{align}
which is the first assertion of the theorem. Note further that by Proposition~\ref{prop:monotonicity} the first inequality in~\eqref{eq:undNochmal} is strict if $H$ does not connect boundary vertices of $G$, i.e., if at least one of the edges $e_1$ and $e_2$ is not incident to a boundary vertex. Moreover, note that the first inequality in~\eqref{eq:einmal} is strict if $H$ contains edges additional to $e_1$ and $e_2$, i.e., if $e_1$ and $e_2$ are not incident to a joint vertex, and that the second inequality in~\eqref{eq:einmal} is strict if the arithmetic mean of $L (e_1)$ and $L (e_2)$ is not equal to the arithmetic mean of all edge lenghts, i.e., whenever $G$ is not equilateral. Summing up, equality in~\eqref{eq:undNochmal} is only possible if $G$ is equilateral and $e_1$ and $e_2$ are boundary edges which are incident to a joint vertex. But in the equilateral case the above reasoning is true for any choice of distinct edges $e_1, e_2$ of $G$. Hence equality in~\eqref{eq:undNochmal} is only possible if $G$ is equilateral and any two edges in $G$ are boundary edges which are incident to a joint vertex, that is, if $G$ is an equilateral star graph. It follows from the explicit representation of the eigenvalues of an equilateral star in Example~\ref{ex:star} that for $k = 1$ equality holds for any equilateral star and that for $k \geq 2$ this is only true if $G$ is an equilateral star consisting of only two edges. This completes the proof of the theorem.
\end{proof}

For the spectral gap of a finite tree we obtain the following corollary.

\begin{corollary}\label{cor:spectralGap}
For each finite, connected tree with $|E| \geq 2$ the spectral gap satisfies
\begin{align*}
 \lambda_2 (G) \leq \frac{|E|^2 \pi^2}{4 L (G)^2}.
\end{align*}
Moreover, equality holds if and only if $G$ is any equilateral star graph.
\end{corollary}

Next we come to estimates involving the diameter of the graph. Recall that the diameter $\diam (G)$ of a graph $G$ is defined as
\begin{align*}
 \diam (G) = \sup \left\{ \dist (x, y) : x, y \in G \right\},
\end{align*}
where $\dist (x, y)$ is the distance of two arbitrary points on the graph with respect to the metric induced by the parametrization of the edges. In particular, for a tree the diameter coincides with the length of the longest path between two boundary vertices.

It was shown in~\cite{KKMM16} that among all finite graphs of a given diameter there is no maximizer for the spectral gap: the authors of~\cite{KKMM16} construct a sequence of graphs $G_n$ of a fixed diameter such that $\lambda_2 (G_n) \to \infty$ as $n \to \infty$. However, the situation is different within the class of finite trees, which can be observed in the following easy way. 

\begin{theorem}\label{thm:diameter}
Let $G$ be a finite, connected tree. Then
\begin{align*}
 \lambda_{k + 1} (G) \leq \frac{k^2 \pi^2}{\diam (G)^2}
\end{align*}
holds for all $k \in \N$.
\end{theorem}

\begin{proof}
As $G$ is a tree, the diameter is attained by a path $H$ through $G$ which connects two boundary vertices. Moreover, $G$ can be obtained by attaching pendant trees to vertices of $H$. Thus Proposition~\ref{prop:monotonicity} implies
\begin{align*}
 \lambda_{k + 1} (G) \leq \lambda_{k + 1} (H) = \frac{k^2 \pi^2}{L (H)^2} = \frac{k^2 \pi^2}{\diam (G)^2}
\end{align*}
for all $k \in \N$; cf.\ Example~\ref{ex:path}. This is the assertion of the theorem.
\end{proof}

We remark that equality in Theorem~\ref{thm:diameter} holds for each path graph and that for $k = 1$ equality is also satisfied for each equilateral star; cf.~Example~\ref{ex:star}.

\section{Estimates involving Dirichlet eigenvalues}

In this section we prove estimates for the eigenvalues of $- \Delta_G$ for a tree $G$ in terms of the Dirichlet eigenvalues of $G$. For the following theorem let $- \Delta_G^{\rm D}$ be the selfadjoint Dirichlet Laplacian in $L^2 (G)$, i.e.,
\begin{align*}
 (- \Delta_G^{\rm D} f)_e & = - f_e'', \quad e \in E,\\
 \dom (- \Delta_G^{\rm D}) & = \big\{f \in \widetilde H^2 (G) \cap H^1 (G) : f (v) = 0~\text{for all}~v \in V \big\},
\end{align*}
and denote by $\lambda_1^{\rm D} (G) \leq \lambda_2^{\rm D} (G) \leq \dots$ its eigenvalues, counted with multiplicities. Note that for any given tree $G$ these eigenvalues can be computed explicitly. They are given by the numbers
\begin{align}\label{eq:DirichletEV}
 \frac{m^2 \pi^2}{L (e)^2}, \quad e \in E, m = 1, 2, \dots, 
\end{align}
where eigenvalues appearing several times in~\eqref{eq:DirichletEV} have respective multiplicities. In particular, the eigenvalues of $- \Delta_G^{\rm D}$ depend only on the edge lengths of $G$ and ignore the geometry of the graph. Note that variational inequalities immediately imply
\begin{align*}
 \lambda_k (G) \leq \lambda_k^{\rm D} (G)
\end{align*}
for all $k \in \N$. This bound is not sharp even in the case of a graph consisting of a single interval of length $L$, where $\lambda_k (G) = (k - 1)^2 \pi^2 / L^2$ but $\lambda_k^{\rm D} (G) = k^2 \pi^2 / L^2$. In fact, for any tree it can be improved as follows; see also~\cite[Lemma~4.5]{BBW15}. 

\begin{theorem}\label{thm:GD}
Let $G$ be a finite, connected tree. Then
\begin{align}\label{eq:Friedlander}
 \lambda_{k + 1} (G) \leq \lambda_k^{\rm D} (G)
\end{align}
holds for all $k \in \N$. Moreover, the inequality~\eqref{eq:Friedlander} is strict for all $k \in \N$ if and only if $G$ contains two edges with rationally independent lengths.
\end{theorem}

\begin{proof}
Let $k \in \N$ be arbitrary and let $\mu = \lambda_k^{\rm D} (G) > 0$. Moreover, define
\begin{align*}
 F = \spann \bigcup_{\lambda \leq \mu} \ker \big(- \Delta_G^{\rm D} - \lambda \big) \subset H^1 (G).
\end{align*}
Then $\dim F \geq k$ and
\begin{align}\label{eq:formMu}
 \int_G |f'|^2 d x \leq \mu \int_G |f|^2 d x, \quad f \in F.
\end{align}
Furthermore, let us define a function $g : G \to \C$ in the following way. Fix an arbitrary boundary vertex $v_0$ of $G$ and assume without loss of generality that the edges of $G$ are parametrized in the direction away from $v_0$, i.e., $\dist (o (e), v_0) < \dist (t (e), v_0)$ for each $e \in E$. For each $e \in E$ define
\begin{align}\label{eq:g}
 g_e (x) = e^{i \sqrt{\mu} (\dist (o (e), v_0) + x)}, \quad x \in [0, L (e)].
\end{align}
Then for each edge $e \in E$ we have 
\begin{align*}
 g_e (0) = e^{i \sqrt{\mu} \dist (o (e), v_0)} \quad \text{and} \quad g_e (L (e)) = e^{i \sqrt{\mu} \dist (t (e), v_0)},
\end{align*}
which implies that $g$ is continuous on $G$. Thus $g \in H^1 (G)$. Clearly, $g$ satisfies
\begin{align}\label{eq:wichtig}
 g' = i \sqrt{\mu} g \quad \text{and} \quad - g'' = \mu g
\end{align}
inside each edge.

Let now $f \in F$, $\eta \in \C$ and $h \in \ker (- \Delta_G - \mu)$. (Note that the latter space may be trivial.) Then integration by parts yields
\begin{align*}
 \int_G |f' + \eta g' + h'|^2 d x & = \int_G \big( |f'|^2 + |\eta g'|^2 + |h'|^2 \big) d x + 2 \Real \int_G (\eta g' + h') \overline{f'} d x \\
 & \quad + 2 \Real \int_G \eta g' \overline{h'} d x \\
 & = \int_G \big( |f'|^2 + |\eta g'|^2 + |h'|^2 \big) d x + 2 \Real \int_G (- \eta g'' - h'') \overline{f} d x \\
 & \quad + 2 \Real \int_G \eta g \overline{(- h'')} d x,
\end{align*}
where we have used $f (v) = \partial_\nu h (v) = 0$ for all $v \in V$. With the help of~\eqref{eq:formMu} and~\eqref{eq:wichtig} we conclude
\begin{align}\label{eq:jetztGehtsLos}
\begin{split}
 \int_G |f' + \eta g' + h'|^2 d x & \leq \mu \Big( \int_G \big( |f|^2 + |\eta g|^2 + |h|^2 \big) d x + 2 \Real \int_G (\eta g + h) \overline{f} d x \\
 & \qquad + 2 \Real \int_G \eta g \overline{h} d x \Big) \\
 & = \mu \int_G |f + \eta g + h|^2 d x.
\end{split}
\end{align}
Note that $|g (x)| = 1$ for all $x \in G$ and, hence, $g \notin F$. Thus~\eqref{eq:jetztGehtsLos} with $h = 0$ implies
\begin{align*}
 N_G ([0, \mu]) \geq \dim (F + \spann \{g\}) \geq k + 1;
\end{align*}
cf.~\eqref{eq:numberEV}. As $\mu = \lambda_k^{\rm D} (G)$, this implies the first assertion of the theorem.

Let us now come to the case that there exist two edges in $G$ with rationally independent edge lengths. Observe first that
\begin{align}\label{eq:linUnab}
 (F + \spann \{g\}) \cap \ker (- \Delta_G - \mu) = \{0\}.
\end{align}
Indeed, assume that for some $f \in F$ and $\eta \in \C$ we have $h := f + \eta g \in \ker (- \Delta_G - \mu)$. Then $f \in \ker (- \Delta_G^{\rm D} - \mu)$ and $f' + \eta g'$ vanishes on each boundary vertex of $G$. If there exists a boundary vertex $v$ with edge $e$ incident to $v$ such that $\sqrt{\mu} \notin \frac{\pi}{L (e)} \N$ then $f$ vanishes on $e$ and thus the derivative of $\eta g$ must vanish at $v$, which is only possible if $\eta = 0$. If, conversely, $\sqrt{\mu} \in \frac{\pi}{L (e)} \N$ for each boundary edge $e$ then the restriction of $f + \eta g$ to the tree $G'$ obtained from $G$ by removing all boundary edges belongs to $\ker (- \Delta_{G'} - \mu)$. Again we conclude $\eta = 0$ if $\sqrt{\mu} \notin \frac{\pi}{L (e)} \N$ for some boundary edge $e$ of $G'$ or we remove all boundary edges from $G'$. Since $G$ is a finite tree and contains two edges with rationally independent edges, after finitely many repetitions of this procedure we arrive at a boundary edge $e$ of a subgraph of $G$ such that $\sqrt{\mu} \notin \frac{\pi}{L (e)} \N$ and obtain $\eta = 0$, i.e., $f = h \in \ker (- \Delta_G - \mu)$. But this implies $f = 0$. Indeed, $f$ and $f'$ must both vanish at each boundary vertex of $G$ and hence must be zero identically on each boundary edge. Applying the standard vertex conditions, the same holds for all boundary edges of the tree obtained from $G$ by removing all boundary edges and then, successively, for each edge of $G$. As $G$ is a tree this can be repeated until no edges are left; hence $h = f = 0$ and we have shown~\eqref{eq:linUnab}. Using this and~\eqref{eq:jetztGehtsLos} we obtain
\begin{align*}
 N_G ([0, \mu]) \geq \dim F + 1 + \dim \ker (- \Delta_G - \mu) \geq k + 1 + \dim \ker (- \Delta_G - \mu),
\end{align*}
which leads to
\begin{align*}
 N_G ( [0, \mu) ) = N_G ([0, \mu]) - \dim \ker (- \Delta_G - \mu) \geq k + 1.
\end{align*}
As $\mu = \lambda_k^{\rm D} (G)$ this yields $\lambda_{k + 1} (G) < \lambda_k^{\rm D} (G)$.

Let now the edge lengths of $G$ be pairwise rationally dependent. Then there exists a real $x > 0$ such that $x L (e) \in \N$ for each $e \in E$. For each $e$ denote by $T_e^{\rm N}$ the selfadjoint operator $- \frac{d^2}{d x^2}$ in $L^2 (e)$ subject to Neumann boundary conditions and let $- \Delta_G^{\rm N}$ be the direct sum of the operators $T_e^{\rm N}$. Then $- \Delta_G^{\rm N}$ is selfadjoint in $L^2 (G)$ and its eigenvalues, counted with multiplicities, are given by
\begin{align}\label{eq:minMaxNeumann}
 \lambda_k^{\rm N} (G) = \min_{\substack{F \subset \widetilde H^1 (G) \\ \dim F = k}} \max_{\substack{f \in F \\ f \neq 0}} \frac{\int_G |f'|^2 d x}{\int_G |f|^2 d x}, \quad k \in \N.
\end{align}
On the one hand,~\eqref{eq:minMaxNeumann} and the inequality~\eqref{eq:Friedlander} obtained in the first part of the proof yield
\begin{align}\label{eq:vonWegen}
 \lambda_{k + 1}^{\rm N} (G) \leq \lambda_{k + 1} (G) \leq \lambda_k^{\rm D} (G) = \lambda_{k + |E|}^{\rm N} (G)
\end{align}
for each $k \in \N$. On the other hand, note that for each $e \in E$ the number $x^2 \pi^2 = \frac{(x L (e))^2 \pi^2}{L (e)^2}$ is an eigenvalue of $T_e$ and $T_e$ has precisely $x L (e)$ eigenvalues strictly below $x^2 \pi^2$. Hence $- \Delta_G^{\rm N}$ has precisely $\sum_{e \in E} (x L (e)) = x L (G)$ eigenvalues strictly below $x^2 \pi^2$ and $x^2 \pi^2$ is eigenvalue of $- \Delta_G^{\rm N}$ with multiplicity $|E|$. Thus
\begin{align}\label{eq:vielGleich}
 \lambda_{x L (G) + 1}^{\rm N} (G) = \dots = \lambda_{x L (G) + |E|}^{\rm N} (G) = x^2 \pi^2.
\end{align}
Comparing~\eqref{eq:vonWegen} with~\eqref{eq:vielGleich} implies equality in~\eqref{eq:Friedlander} for $k = x L (G)$. This completes the proof.
\end{proof}

A particularly simple consequence of Theorem~\ref{thm:GD} is the following estimate; here $L_{\max}$ denotes the length of the longest edge in $G$.

\begin{corollary}\label{cor:explicit}
Let $G$ be a finite, connected tree. Then the estimate
\begin{align}\label{eq:firstEstimate}
 \lambda_{k + 1} (G) \leq \frac{k^2 \pi^2}{L_{\max}^2}
\end{align}
holds for all $k \in \N$. If $G$ contains two edges with rationally independent lengths then the inequality~\eqref{eq:firstEstimate} is strict for all $k \in \N$.
\end{corollary}

We remark that the estimate~\eqref{eq:firstEstimate} (and thus also the estimate in Theorem~\ref{thm:GD}) does not hold in general for a graph with cycles. For instance, for the graph formed by one vertex and one loop of length $L$ one has $\lambda_2 (G) = \frac{4 \pi^2}{L^2}$, which contradicts the estimate~\eqref{eq:firstEstimate} for $k = 1$ in this case. Furthermore, note that equality holds in~\eqref{eq:firstEstimate} if and only if $G$ is a single interval; this can be seen with the help of the monotonicity principle in Proposition~\ref{prop:monotonicity}. The following example shows that for $k = 1$ the estimate~\eqref{eq:firstEstimate} is sharp for $|E| = 3$. Analogous examples can be easily constructed for any $|E| \geq 2$.

\begin{example}
Consider a star graph $G$ consisting of three edges $e_1, e_2, e_3$ with lengths $L (e_1) = L (e_2) = \pi/n$ and $L (e_3) = \pi$ with a given, fixed even number $n \in \N$. A simple computation gives the eigenvalues of $- \Delta_G$. In fact, $\lambda \geq 0$ is an eigenvalue if and only if either
\begin{align}\label{eq:aha}
 \sqrt{\lambda} = n \Big(k + \frac{1}{2} \Big), \quad k \in \N,
\end{align} 
or 
\begin{align}\label{eq:interessant}
 \cos \Big(\sqrt{\lambda} \frac{\pi}{n} \Big) \sin \big(\sqrt{\lambda} \pi \big) + 2 \sin \Big(\sqrt{\lambda} \frac{\pi}{n} \Big) \cos \big(\sqrt{\lambda} \pi \big) = 0.
\end{align}
For any $n \geq 2$ the eigenvalues obtained through~\eqref{eq:aha} satisfy $\lambda \geq 1$. On the other hand, it can be shown that the equation~\eqref{eq:interessant} has exactly one solution $\lambda (n)$ in $(0, 1)$, which satisfies $\lambda (n) \to 1$ as $n \to \infty$. Thus 
\begin{align*}
 \lambda_2 (G) = \lambda (n) \to 1 = \frac{\pi^2}{L_{\max}^2} \quad \text{as} \quad n \to \infty.
\end{align*}
\end{example}

\bibliographystyle{amsplain}

\end{document}